\documentclass[reqno]{amsart} \usepackage{amscd}
\usepackage{epsf}
\newtheorem{theorem}{Theorem}[section]
\newtheorem{proposition}[theorem]{Proposition}

\theoremstyle{remark} 

\begin{document}

\title[Ramanujan's identities, minimal surfaces and solitons]{Ramanujan's identities, minimal surfaces and solitons }

\author{Rukmini Dey}
\address{School of Mathematics\\
Harish Chandra Research Institute\\
Allahabad, 211019, India\\
rkmn@mri.ernet.in}

\begin{abstract}
Using Ramanujan's identites and the Weierstrass-Enneper representation of minimal surfaces, and the analogue for Born-Infeld solitons,  we derive further non-trivial identities.
\end{abstract}

\maketitle

\section{Introduction}

Using some of Ramanujan's identites and the Weierstrass-Enneper representation 
of minimal surfaces, and the analogue for Born-Infeld solitons, we obtain non-trivial identities. They have the  feature that most of them depend on just one 
complex parameter.  Ramanujan's idenities 
were first 
used in the context of minimal surfaces perhaps by  Kamien, ~\cite{K}. 

The identities we obtain in this paper are:

1)  For $\zeta \neq \pm 1, \pm i $, 

\begin{eqnarray*}
 & &   {\rm Re} {\rm ln}(\frac{ 1 + \zeta^2}{  1 - \zeta^2})   \\
& & =  \sum_{k=1}^{\infty} {\rm ln} (\frac{ - {\rm Im} {\rm ln} (\frac{1 + \zeta}{1-\zeta})  -(k - \frac{1}{2}) \pi}{  2 {\rm Re} {\rm tan}^{-1}(\zeta) - (k - \frac{1}{2})\pi}) + \sum_{k=1}^{\infty} {\rm ln} ( \frac{ - {\rm Im} {\rm ln} (\frac{1 + \zeta}{1-\zeta}) + (k - \frac{1}{2}) \pi}{ 2 {\rm Re} {\rm tan}^{-1}(\zeta)    + (k - \frac{1}{2})\pi})
\end{eqnarray*}

2) For $r, s \neq \pm 1$,

\begin{eqnarray*}
 & &   \frac{1}{2} {\rm ln} ( \frac{1 + r^2}{1-r^2} ) +  \frac{1}{2} {\rm ln} ( \frac{1 + s^2}{1-s^2} )  \\
& &  =  \sum_{k=1}^{\infty} {\rm ln} (\frac{ (k - \frac{1}{2}) \pi -i ( - {\rm tanh }^{-1} (r) + {\rm tan h }^{-1} (s)  ) }{ (k - \frac{1}{2})\pi -  (  + {\rm tan}^{-1} (r) + {\rm tan}^{-1} (s) )  }) \\
& & + \sum_{k=1}^{\infty} {\rm ln} (  \frac{ (k - \frac{1}{2}) \pi + i ( - {\rm tanh }^{-1} (r) + {\rm tan h }^{-1} (s)  ) }{ (k - \frac{1}{2})\pi  + {\rm tan}^{-1} (r) + {\rm tan}^{-1} (s)})
\end{eqnarray*}

3) For $\zeta \neq 0$,

\begin{eqnarray*}
 & &   -\frac{\pi}{2} +  \rm{Im} (ln \zeta) -  {\rm tan}^{-1} [ {\rm tan h} ( \frac{1}{2} \rm{Re} ( \zeta - \frac{1}{\zeta}) ){\rm cot} ( -\frac{1}{2} \rm{Im} ( \zeta + \frac{1}{\zeta})  ) ]   \\
  &=& - \sum^{k= \infty}_{k = 1} {\rm tan }^{-1} ( \frac{\frac{1}{2} \rm{Re} ( \zeta - \frac{1}{\zeta})  }{-\frac{1}{2} \rm{Im} ( \zeta + \frac{1}{\zeta}) + k \pi})
- \sum^{k= \infty}_{k = 1} {\rm tan }^{-1} ( \frac{\frac{1}{2} \rm{Re} ( \zeta - \frac{1}{\zeta})}{ -\frac{1}{2} \rm{Im} ( \zeta + \frac{1}{\zeta})- k \pi})
\end{eqnarray*}

4) 
 For $\zeta \neq \pm e^{\pm i \frac{\pi}{4}}$, there exists two integers $m,n$ 
such that    
\begin{eqnarray*} 
& & (n \pi +  \rm{Im} \rm{ln}(\frac{\zeta^2 + \sqrt{2} \zeta + 1}{\zeta^2 - \sqrt{2} \zeta + 1}) \\
&=&  \sum^{k= \infty}_{k = -\infty} {\rm tan }^{-1} ( \frac{\sqrt{2} \rm{Im} \rm{tan}^{-1} \zeta^2  }{m \pi +2 \rm{Re} \arctan(\frac{\zeta \sqrt{2}}{1 - \zeta^2})   +  k \pi})
\end{eqnarray*}

\section{The Identities}

 Recall the Weierstrass-Enneper representation  ~\cite{N} (page 147),  ~\cite{D1}, 
of minimal surfaces, namely, 
 in the neighborhood of a
nonumbilic interior point, any minimal surface can
be represented  as follows, 
\begin{eqnarray*}
x(\zeta) &=& \tilde{x}_0 + \rm{Re} \int_{\zeta_0}^{\zeta} (1-w^2) R(w) \,d w\\
y(\zeta) &=& \tilde{y}_0 + \rm{Re} \int_{\zeta_0}^{\zeta} i (1 + w^2) R(w) \,d w\\
z(\zeta) &=& \tilde{z}_0 + \rm{Re} \int_{\zeta_0}^{\zeta} 2 w R(w) \,d w
\end{eqnarray*}

Here $\zeta$ is a complex parameter and $R(w)$ is a meromorphic function. 
This is an isothermal representation (w.r.t.  $\zeta_1$ and $\zeta_2$ where 
$\zeta = \zeta_1 + i \zeta_2.$).   In ~\cite{D1} and ~\cite{D2}, we show, using hodographic coordinates,  how to 
compute the $R(w)$ for minimal surfaces which are given locally by a graph  $z = z(x,y)$.

Recall, that the Gaussian curvature is given by $ K = -4 |R(w)|^{-2} ( 1 + |w|^2)^{-4}.$ Thus the umbilical points correspond to the poles of $R$, ~\cite{N} (pages 148 and 472). This is precisely where the representation fails.

\subsection{The first identity}

We have Ramanujan's identity, ~\cite{R}, Example $(1)$ page $38$, where 
$X$, $A$ are complex, $A$ is not an odd multiple of $\pi /2$:

$\frac{{\rm cos}(X+A)}{{\rm cos} (A)} = \Pi_{k=1}^{\infty} \{ ( 1-\frac{X}{(k - \frac{1}{2} \pi) -A} ) ( 1 + \frac{X}{(k - \frac{1}{2} \pi) +A} )\}$.

We take ln on both sides, to get:

\begin{eqnarray*}
&& {\rm ln} (\frac{{\rm cos} (X+A)}{{\rm cos} (A)}) \\
&=& \sum_{k=1}^{\infty} {\rm ln} (1-\frac{X}{(k - \frac{1}{2}) \pi - A} ) + \sum_{k=1}^{\infty} {\rm ln} ( 1 + \frac{X}{(k - \frac{1}{2}) \pi +A} )\\
&=&   \sum_{k=1}^{\infty} {\rm ln} (\frac{ (k - \frac{1}{2} \pi ) -(X + A)}{(k - \frac{1}{2}) \pi - A} ) + \sum_{k=1}^{\infty} {\rm ln} (  \frac{ (k - \frac{1}{2}) \pi + (X + A) }{(k - \frac{1}{2}) \pi +A} )
\end{eqnarray*}

The Scherk's second surface is given by 
$ z = {\rm ln} (\frac{{\rm cos}(t)}{{\rm cos}(x)})$ (see Nitsche, equation number $(27)$, page 71).

Let $X+A = y $ and $A=x$ in Ramanujan's identity.

Then, if $x$ is not an odd multiple of $\frac{\pi}{2}$, we have, 

\begin{eqnarray*}
{\rm ln} (\frac{{\rm cos}(y)}{{\rm cos}(x)}) &=& \sum_{k=1}^{\infty} {\rm log} (\frac{ y-(k - \frac{1}{2}) \pi )}{x - (k - \frac{1}{2})\pi}) + \sum_{k=1}^{\infty} {\rm log} (  \frac{y + (k - \frac{1}{2}) \pi}{x + (k - \frac{1}{2})\pi} )
\end{eqnarray*}

Now since the left hand side is the height function of a minimal surface, we 
can use its Weierstrass-Enneper representation.

$R(w) = \frac{2}{(1 - w^4)}$ leads to the Scherk's second minimal surface,  $z = {\rm ln} (\frac{{\rm cos} (y)}{{\rm cos}(x)})$,  ~\cite{N}, (page 71, 148).
This non-parametric representation is valid in the domain :

$\{ (x,y): | \sqrt{2}(x-y) - 4m\pi| < \pi,  | \sqrt{2}(x+y) - 4n \pi| < \pi \}$

where $m, n = 0, \pm 1, \pm 2,....$.

If we perform the integrals given by the W-E representation formula, we get 
\begin{eqnarray*}
x(\zeta) &=& x_0  + 2 {\rm Re} {\rm tan}^{-1}(\zeta)\\
y(\zeta) &=& y_0 - {\rm Im} {\rm ln} (\frac{1 + \zeta}{1-\zeta}) \\
z(\zeta) &=& z_0 + {\rm Re} {\rm ln} (\frac {1 + \zeta^2}{ 1 - \zeta^2})  .
\end{eqnarray*}

 If we take $x_0=y_0=z_0=0$ we get: 
\begin{eqnarray*}
x(\zeta) &=&  2 {\rm Re} {\rm tan}^{-1}(\zeta)  \\
y(\zeta) &=& - {\rm Im} {\rm ln} (\frac{1 + \zeta}{1-\zeta}) \\
z(\zeta) &=&  {\rm Re} {\rm ln} (\frac {1 + \zeta^2}{ 1 - \zeta^2}) . 
\end{eqnarray*}

Using  the fact that $ {\rm ln} (Z) = \rm{ln} |Z| + i \theta = \rm{ln} |Z| + i \rm{tan}^{-1} (\frac{\rm{Im} Z}{\rm{Re} Z})$ where $Z= | Z| e^{i \theta}$,   for $Z$ any complex number, one 
can easily check that  in the above parametrization, with $x_0=y_0=z_0=0$,   $$z = {\rm ln} (\frac{{\rm cos} (y)}{{\rm cos}(x)})$$ This parametrization  
 fails precisely  at $\zeta = \pm 1, \pm i$, the umbilical
points of the minimal surface (since these are precisely the poles of $R(w)$).

\begin{proposition}
Our first identity for $\zeta \neq \pm 1, \pm i$ is the following:
\begin{eqnarray*}
 & &   {\rm Re} {\rm ln}(\frac{ 1 + \zeta^2}{  1 - \zeta^2})   \\
& & =  \sum_{k=1}^{\infty} {\rm ln} (\frac{ - {\rm Im} {\rm ln} (\frac{1 + \zeta}{1-\zeta})  -(k - \frac{1}{2}) \pi}{  2 {\rm Re} {\rm tan}^{-1}(\zeta) - (k - \frac{1}{2})\pi}) + \sum_{k=1}^{\infty} {\rm ln} ( \frac{ - {\rm Im} {\rm ln} (\frac{1 + \zeta}{1-\zeta}) + (k - \frac{1}{2}) \pi}{ 2 {\rm Re} {\rm tan}^{-1}(\zeta)    + (k - \frac{1}{2})\pi})
\end{eqnarray*}
\end{proposition}

\begin{proof}

Substituting the W-E in Ramanujan's identity, we get:

\begin{eqnarray*}
& &  {\rm Re} {\rm ln}(\frac{ 1 + \zeta^2}{  1 - \zeta^2}) \\ 
& & =    {\rm ln} (\frac{{\rm cos}(- {\rm Im} {\rm ln} (\frac{1 + \zeta}{1-\zeta}) )}{{\rm cos}( 2 {\rm Re} {\rm tan}^{-1}(\zeta) )}) \\
&=& \sum_{k=1}^{\infty} {\rm ln} (\frac{- {\rm Im} {\rm ln} (\frac{1 + \zeta}{1-\zeta})  -(k - \frac{1}{2}) \pi}{ 2 {\rm Re} {\rm tan}^{-1}(\zeta) - (k - \frac{1}{2})\pi}) + \sum_{k=1}^{\infty} {\rm ln} (\frac{- {\rm Im} {\rm ln} (\frac{1 + \zeta}{1-\zeta}) + (k - \frac{1}{2}) \pi}{ 2 {\rm Re} {\rm tan}^{-1}(\zeta)    + (k - \frac{1}{2})\pi})
\end{eqnarray*}

Thus we get our first identity.

\end{proof}

Notice that the transformations  $y \rightarrow -y$ or $x \rightarrow -x$ or $x, y \rightarrow -x, -y$ give the same height 
function $z = {\rm ln} (\frac{{\rm cos} (y)}{{\rm cos}(x)})$ and hence give new identities or different ways of writing the same identities. 

For instance, $y \rightarrow -y$ gives:

\begin{eqnarray*}
 & &  {\rm Re} {\rm ln} (\frac{ 1 + \zeta^2}{  1 - \zeta^2})  \\
& & = \sum_{k=1}^{\infty} {\rm ln} (\frac{ + {\rm Im} {\rm ln} (\frac{1 + \zeta}{1-\zeta})  -(k - \frac{1}{2}) \pi )}{  + 2 {\rm Re} {\rm tan}^{-1}(\zeta) - (k - \frac{1}{2})\pi}) + \sum_{k=1}^{\infty} {\rm ln} ( \frac{ + {\rm Im} {\rm ln} (\frac{1 + \zeta}{1-\zeta}) + (k - \frac{1}{2}) \pi}{ + 2 {\rm Re} {\rm tan}^{-1}(\zeta)    + (k - \frac{1}{2})\pi})
\end{eqnarray*}

\subsection{The second identity}

Notice that the minimal surface equation is just the wick rotated 
Born Infeld equation. We exploited this fact in ~\cite{D1},  ~\cite{D2}.

If the minimal surface is given by $z=z(x,t)$ locally,
then it follows the equation 

$$(1+ z_t^2) z_{xx} - 2 z_x z_t z_{xt} + (1 + z_x^2) z_{tt} =0$$

The Born-Infeld solitons follow the equation

$$(1- z_t^2) z_{xx} + 2 z_x z_t z_{xt} - (1 + z_x^2) z_{tt}
=0$$

which can be obtained from the first equation by  wick rotation, namely, 
$t \rightarrow it$.

Thus, if $z = {\rm ln} (\frac{{\rm cos}(t)}{{\rm cos}(x)})$ is a solution of the 
minimal surface equation, then
$z = {\rm ln} (\frac{{\rm cos}(it)}{{\rm cos}(x)}) ={\rm ln} (\frac{{\rm cosh }(t)}{{\rm cos}(x)}) $ is a solution of the B-I equation.

(We let $x$, $t$ and $z$ to be complex.)

We can find the analogue of the Weierstrass-Enneper representation of the B-I
solitons in Whitham, ~\cite{W}, page 617, (based on a method by Barbishov and Chernikov , ~\cite{BC}).

Following their calculation for 
$z  ={\rm ln} (\frac{{\rm cosh }(t)}{{\rm cos}(x)}) $ 
we get $z_x = {\rm tan} x$, $z_t = {\rm tanh}(t)$. 
$u = \frac{z_x - z_t}{2}$ and $v = \frac{z_x + z_t}{2}$.

Let $r =  \frac{\sqrt{1 + 4 u v } - 1}{2 v}$ and $s = \frac{\sqrt{1 + 4 u v } - 1}{2 u}$. 

Then $u = \frac{r}{1 -rs}$ and $v = \frac{s}{1-rs}$. 

Then ${\rm tan}(x) = \frac{r + s }{1-rs}$ and ${\rm tanh}(t) = \frac{s-r}{1-rs}$.

In other words, 

$x = x_0 + {\rm tan}^{-1} (r) + {\rm tan}^{-1} (s)$

$t = t_0 - {\rm tanh }^{-1} (r) + {\rm tan h }^{-1} (s)$

$z= z_0 + \frac{1}{2} {\rm ln} ( \frac{1 + r^2}{1-r^2} ) +  \frac{1}{2} {\rm ln} ( \frac{1 + s^2}{1-s^2} ) $

Here, $F(r) = {\rm tan}^{-1} (r) +  {\rm tanh}^{-1} (r)$
and $G(s) = {\rm tan}^{-1} (s) +  {\rm tanh}^{-1} (s)$ , (notation as in ~\cite{W}). Also, $r$ and $s$ are complex parameters, since $x$, $t$ and $z$ are complex variables.

Again taking $x_0=y_0=z_0 = 0$, we get a parametrization for the complex soliton
$z= \rm{ln} (\frac{cosh(t)}{cos(x)})$. (Easy to check this). 

$x =  {\rm tan}^{-1} (r) + {\rm tan}^{-1} (s)$

$t = - {\rm tanh }^{-1} (r) + {\rm tan h }^{-1} (s)$

$z=   \frac{1}{2} {\rm ln} ( \frac{1 + r^2}{1-r^2} ) +  \frac{1}{2} {\rm ln} ( \frac{1 + s^2}{1-s^2} ) $

Note that if one takes  a special relation between the parameters $r = \bar{s} = \zeta$ and $y = it$ then we get back the  parametrizaton of $z = \rm{ln} (\frac{cos y }{cos x})$.

\begin{proposition}
We have our second identity, i.e. for $r, s \neq \pm 1$:
\begin{eqnarray*}
 & &   \frac{1}{2} {\rm ln} ( \frac{1 + r^2}{1-r^2} ) +  \frac{1}{2} {\rm ln} ( \frac{1 + s^2}{1-s^2} )  \\
& &  =  \sum_{k=1}^{\infty} {\rm ln} (\frac{ (k - \frac{1}{2}) \pi -i ( - {\rm tanh }^{-1} (r) + {\rm tan h }^{-1} (s)  ) }{ (k - \frac{1}{2})\pi -  (   {\rm tan}^{-1} (r) + {\rm tan}^{-1} (s) )  }) \\
& & + \sum_{k=1}^{\infty} {\rm ln} (  \frac{ (k - \frac{1}{2}) \pi + i ( - {\rm tanh }^{-1} (r) + {\rm tan h }^{-1} (s)  ) }{ (k - \frac{1}{2})\pi  + {\rm tan}^{-1} (r) + {\rm tan}^{-1} (s)})
\end{eqnarray*}

 \end{proposition}

\begin{proof}

By the Ramanujan's identity (which we used to get our first identity) we have:

\begin{eqnarray*}
{\rm ln} (\frac{{\rm cosh}(t)}{{\rm cos}(x)}) &=& \sum_{k=1}^{\infty} {\rm ln} (\frac{ (k - \frac{1}{2}) \pi -it}{ (k - \frac{1}{2})\pi - x}) + \sum_{k=1}^{\infty} {\rm ln} (  \frac{ (k - \frac{1}{2}) \pi + it }{ (k - \frac{1}{2})\pi + x} )
\end{eqnarray*}

Substituting in Ramanujan's identity, we get: 

\begin{eqnarray*}
& &   \frac{1}{2} {\rm ln} ( \frac{1 + r^2}{1-r^2} ) +  \frac{1}{2} {\rm ln} ( \frac{1 + s^2}{1-s^2} ) \\
& & = {\rm ln} (\frac{{\rm cosh}( - {\rm tanh }^{-1} (r) + {\rm tan h }^{-1} (s)   )}{{\rm cos}( {\rm tan}^{-1} (r) + {\rm tan}^{-1} (s) )})\\
& & = \sum_{k=1}^{\infty} {\rm ln} (\frac{ (k - \frac{1}{2}) \pi -i (  - {\rm tanh }^{-1} (r) + {\rm tan h }^{-1} (s)  ) }{ (k - \frac{1}{2})\pi -  (  {\rm tan}^{-1} (r) + {\rm tan}^{-1} (s) )  }) \\
& & +  \sum_{k=1}^{\infty} {\rm ln} (  \frac{ (k - \frac{1}{2}) \pi + i ( - {\rm tanh }^{-1} (r) + {\rm tan h }^{-1} (s)  ) }{ (k - \frac{1}{2})\pi + {\rm tan}^{-1} (r) + {\rm tan}^{-1} (s)})
\end{eqnarray*}

Thus we have our second identity.
\end{proof}

\subsection{The third identity}

  By Ramanujan's identity,~\cite{R} page 39, ~\cite{K}, for $A$ and $B$ real,
we have,

\begin{eqnarray*} 
{\rm tan}^{-1} [ {\rm tan h}  A {\rm cot} B ] = \sum^{k= \infty}_{k = -\infty} {\rm tan }^{-1} ( \frac{A}{B + k \pi}). 
\end{eqnarray*}

Therefore, 
\begin{eqnarray*} 
{\rm tan}^{-1} ( {\rm tanh}t {\rm cot}x  ) = \sum^{k= \infty}_{k = -\infty} {\rm tan }^{-1} ( \frac{t}{x + k \pi}). 
\end{eqnarray*}

Separating the $k=0$ term, which gives the  the height function of the 
helicoid, we get,

\begin{eqnarray*}
 {\rm tan}^{-1} (\frac{t}{x}) &= &{\rm tan}^{-1} [ {\rm tan h}  t {\rm cot}x  ]\\
  & & - \sum^{k= \infty}_{k = 1} {\rm tan }^{-1} ( \frac{t}{x + k \pi})
- \sum^{k= \infty}_{k = 1} {\rm tan }^{-1} ( \frac{t}{x - k \pi})
\end{eqnarray*}

Using the Weierstrass Enneper representation of the helicoid, ~\cite{D2}, we get

$x = -\frac{1}{2} \rm{Im} ( \zeta + \frac{1}{\zeta})$

$t = \frac{1}{2} \rm{Re} ( \zeta - \frac{1}{\zeta})$

$z = -\frac{\pi}{2} +  \rm{Im} (ln \zeta) $

{\bf Correction:} In ~\cite{D2}, we missed out $z_0=-\frac{\pi}{2} $.

This representation is invalid at $\zeta =0$. 

Substituting this W-E representation in  $z= \rm{ tan}^{-1} \frac{t}{x}$, we get 
\begin{eqnarray*}
 & &  -\frac{\pi}{2} +  \rm{Im} (ln \zeta) =  {\rm tan}^{-1} (\frac{ \frac{1}{2} \rm{Re} ( \zeta - \frac{1}{\zeta})  }{-\frac{1}{2} \rm{Im} ( \zeta + \frac{1}{\zeta})  }) = \\ & & {\rm tan}^{-1} [ {\rm tan h} ( \frac{1}{2} \rm{Re} ( \zeta - \frac{1}{\zeta}) ){\rm cot} ( -\frac{1}{2} \rm{Im} ( \zeta + \frac{1}{\zeta})  ) ]\\
  & & - \sum^{k= \infty}_{k = 1} {\rm tan }^{-1} ( \frac{\frac{1}{2} \rm{Re} ( \zeta - \frac{1}{\zeta})  }{-\frac{1}{2} \rm{Im} ( \zeta + \frac{1}{\zeta}) + k \pi})
- \sum^{k= \infty}_{k = 1} {\rm tan }^{-1} ( \frac{\frac{1}{2} \rm{Re} ( \zeta - \frac{1}{\zeta})}{ -\frac{1}{2} \rm{Im} ( \zeta + \frac{1}{\zeta})- k \pi})
\end{eqnarray*}

Thus we get our third identity, namely,
\begin{proposition} 
For $\zeta \neq 0$, 
\begin{eqnarray*}
 & &   -\frac{\pi}{2} +  \rm{Im} (ln \zeta) -  {\rm tan}^{-1} [ {\rm tan h} ( \frac{1}{2} \rm{Re} ( \zeta - \frac{1}{\zeta}) ){\rm cot} ( -\frac{1}{2} \rm{Im} ( \zeta + \frac{1}{\zeta})  ) ]   \\
  &=& - \sum^{k= \infty}_{k = 1} {\rm tan }^{-1} ( \frac{\frac{1}{2} \rm{Re} ( \zeta - \frac{1}{\zeta})  }{-\frac{1}{2} \rm{Im} ( \zeta + \frac{1}{\zeta}) + k \pi})
- \sum^{k= \infty}_{k = 1} {\rm tan }^{-1} ( \frac{\frac{1}{2} \rm{Re} ( \zeta - \frac{1}{\zeta})}{ -\frac{1}{2} \rm{Im} ( \zeta + \frac{1}{\zeta})- k \pi})
\end{eqnarray*}
\end{proposition}

\subsection{The fourth identity}
By Ramanujan's identity, ~\cite{R} page 39, ~\cite{K}, for $A$ and $B$ real,

\begin{eqnarray*} 
{\rm tan}^{-1} [ {\rm tan h}  A {\rm cot} B ] = \sum^{k= \infty}_{k = -\infty} {\rm tan }^{-1} ( \frac{A}{B + k \pi}). 
\end{eqnarray*}

The Scherk's first surface is given by 
$$ {\rm tanh}(\frac{z}{a}) = {\rm tan}(\frac{x}{a {\rm cos}(\alpha)}) {\rm tan}(\frac{y}{a {\rm sin}(\alpha)})$$.
This non-parametric representation is valid in the domain
$$ \{ (x,y): | \frac{x}{a {\rm cos}(\alpha)} - \frac{y}{a {\rm sin}(\alpha)} - 2 m a \pi | < \frac{a \pi}{2},  | \frac{x}{a {\rm cos}(\alpha)} + \frac{y}{a {\rm sin}(\alpha)} - 2 n a \pi | < \frac{a \pi}{2} \} $$
where $m, n = 0 \pm 1, \pm 2,...$.

We get
\begin{eqnarray*}
 \frac{x}{a {\rm cos}(\alpha)} &=& {\rm tan}^{-1}({\rm tanh}(\frac{z}{a}) {\rm cot}(\frac{y}{a {\rm sin}(\alpha)})\\
&=&  \sum^{k= \infty}_{k = -\infty} {\rm tan }^{-1} ( \frac{z {\rm sin}(\alpha)}{ y  + a {\rm sin}(\alpha) k \pi}). 
\end{eqnarray*}

By Nitsche, ~\cite{N}, page 148 and page 70, $R(w) = \frac{-2a i {\rm sin}(2 \alpha)}{ 1 + 2 w^2 {\rm cos}(2 \alpha) + w^4} $ in the Weierstrass-Enneper representation with $0 < \alpha < \frac{\pi}{2}$, $a > 0$ , leads to the Scherk's first minimal surface.

Even though one can perform the W-E integrals for a general $\alpha$, we choose $\alpha = \frac{\pi}{4}$.

Performing the integrals, ~\cite{GR}, page 74  and  84, we get:

\begin{eqnarray*}
x(\zeta) &=&  \frac{a}{\sqrt{2}} (x_0 +  \rm{Im} \rm{ln}(\frac{\zeta^2 + \sqrt{2} \zeta + 1}{\zeta^2 - \sqrt{2} \zeta + 1}) \\
y(\zeta) &=&  \frac{a}{ \sqrt{2}}( y_0 +2 \rm{Re} (\arctan(\frac{\zeta \sqrt{2}}{(1 - \zeta^2)})  \\
z(\zeta) &=&  a( z_0 + 2 \rm{Im} \rm{tan}^{-1} \zeta^2)
\end{eqnarray*}

By a suitable choice of $(x_0, y_0, z_0)$ this minimal surface satisfies the equation  $$ {\rm tanh}(\frac{z}{a}) = {\rm tan}(\frac{\sqrt{2} x}{a}) {\rm tan}(\frac{\sqrt{2}y}{a})$$ 

The surface passes through $x_0, y_0, z_0$ at $\zeta=0$.

This representation is invalid at the four points $ \zeta = \pm e^{\pm i \frac{\pi}{4}}$
which correspond to the umbilical points of the minimal surface (poles of $R$).

Substituting in Ramanujan's identity,  we get:

\begin{eqnarray*}
& &  \frac{\sqrt{2}(x_0 +  \rm{Im} \rm{ln} (\frac{\zeta^2 + \sqrt{2} \zeta + 1}{\zeta^2 - \sqrt{2} \zeta + 1}))}{a} \\
&=& {\rm tan}^{-1}({\rm tanh}(\frac{ z_0 +2 \rm{Im} \rm{tan}^{-1} \zeta^2 }{a}) {\rm cot}(\sqrt{2}(\frac{ y_0 +2 \rm{Re}(\arctan(\frac{\zeta \sqrt{2}}{(1 - \zeta^2)})}{a})) \\
&=&  \sum^{k= \infty}_{k = -\infty} {\rm tan }^{-1} ( \frac{z_0 + 2 \rm{Im} \rm{tan}^{-1} \zeta^2 }{\sqrt{2}(y_0 +2 \rm{Re} \arctan(\frac{\zeta \sqrt{2}}{1 - \zeta^2}))   + a k \pi})
\end{eqnarray*}

We take $a= \sqrt{2}$.
To find $x_0, y_0, z_0$ we  try various values of $\zeta$.

First note that if $\zeta = \zeta_1$ any real number, 
we have the identity

$\rm{tan} (x_0) = \rm{tanh} (\frac{z_0}{\sqrt{2}}) \rm{cot}( y_0 + 2 \rm{tan}^{-1}(\frac{\sqrt{2}\zeta_1}{(1-\zeta_1^2)})$ for all $\zeta_1$ real.
This can be true only if $z_0=0$ and $x_0 = n \pi$. 
 
Next we try $\zeta= \zeta_2 $ purely imaginary.

Let $C_1= \rm{Im} \rm{ln}(\frac{\zeta_2^2 + \sqrt{2} \zeta_2 + 1}{\zeta_2^2 - \sqrt{2} \zeta_2 + 1})$, $C_2= 2 \rm{Re} (\arctan(\frac{\zeta_2 \sqrt{2}}{(1 - \zeta_2^2)}) =0$ , $C_3= 2 \rm{Im} \rm{tan}^{-1} \zeta_2^2 =0$.

Then
$\rm{tan}(n \pi +  C_1) = \rm{tanh}(\frac{z_0}{\sqrt{2}}) \cdot \rm{cot} ( y_0)$ or, $\rm{tan}(y_0) \cdot \rm{tan}(n \pi +  C_1) = \rm{tanh}(\frac{z_0}{\sqrt{2}})$.

Since  $z_0 =0$, $y_0 = m \pi$. 

$m, n$ could be fixed by taking $\zeta=\zeta_3, \zeta_4$ two arbitrary complex 
numbers.

Thus we get our fourth identity:
\begin{proposition}
For $\zeta \neq \pm e^{\pm i \frac{\pi}{4}}$, there exists two integers $m,n$ 
such that    
\begin{eqnarray*} 
& & (n \pi +  \rm{Im} \rm{ln}(\frac{\zeta^2 + \sqrt{2} \zeta + 1}{\zeta^2 - \sqrt{2} \zeta + 1}) \\
&=&  \sum^{k= \infty}_{k = -\infty} {\rm tan }^{-1} ( \frac{\sqrt{2} \rm{Im} \rm{tan}^{-1} \zeta^2  }{m \pi +2 \rm{Re} \arctan(\frac{\zeta \sqrt{2}}{1 - \zeta^2}) +  k \pi})
\end{eqnarray*}
\end{proposition}

\section{Correction to a previous paper}

There are corrections to the paper, Dey ~\cite{D1}.

1. The method of deriving the W-E representation adapted in this paper is due to Barbishov and Chernikov ~\cite{BC}
and not Whitham (as erroneously mentioned in the abstract).  In ~\cite{BC}, 
Barbishov and Chernikov develop this  method in the context of Born-Infeld solitons, which is outlined by Whitham  in ~\cite{W}.

2. The method fails precisely when when $\phi_{zz} \phi_{\bar{z} \bar{z}} - (\phi_{{z} \bar{z}})^2 =0$ (as explained in the paper). By a calculation, one shows that 
 $\phi_{zz} \phi_{\bar{z} \bar{z}} - (\phi_{{z} \bar{z}})^2=  (\phi_{xx} \phi_{yy} - \phi_{xy}^2 )$ and thus the method breaks down precisely when $ (\phi_{xx} \phi_{yy} - \phi_{xy}^2 )=0$ , i.e. at the umbilical points. 
This is in accordance with the usual derivation of Weierstrass-Enneper 
representation of minimal surfaces. I had mistakenly mentioned in ~\cite{D2} that they are two different conditions.

\end{document}